\documentclass{article}
\usepackage{graphicx}
\usepackage{mathptmx}
\usepackage{latexsym, amsmath, amsfonts, amssymb,  setspace, yfonts, bm}
\usepackage{amsthm, mathtools}

\newcommand{\pg}{\mathrm{PG}}
\newcommand{\pgl}{\mathrm{PGL}}

\newcommand{\ag}{\mathrm{AG}}

\theoremstyle{plain} \numberwithin{equation}{section}
\newtheorem{theorem}{Theorem}[section]

\newtheorem{lemma}[theorem]{Lemma}

\theoremstyle{definition}
\newtheorem{definition}[theorem]{Definition}

\begin{document}
\bibliographystyle{plain-annote}
\title{Cameron-Liebler line classes}

\author{Morgan Rodgers
  \thanks{This research was supported in part by NSF \#0742434 UCD GK12 Transforming Experiences Project.}\\
  University of Colorado Denver\\
  Department of Mathematical and Statistical Sciences\\
  morgan.rodgers@ucdenver.edu
}
\date{July 1, 2013}

\maketitle

\begin{abstract}

New examples of Cameron-Liebler line classes in $\pg(3,q)$ are given with parameter  $\frac{1}{2}(q^2 -1)$.
These examples have been constructed for many odd values of $q$ using a computer search, by forming a union
of line orbits from a cyclic collineation group acting on the space. While there are many equivalent
characterizations of these objects, perhaps the most significant is that a set of lines $\mathcal{L}$ in $\pg(3,q)$
is a Cameron-Liebler line class with parameter $x$ if and only if every spread $\mathcal{S}$ of the space
shares precisely $x$ lines with $\mathcal{L}$. These objects are related
to generalizations of symmetric tactical decompositions of
$\pg(3,q)$, as well as to subgroups of $\pgl(4,q)$ having equally many orbits on points and lines of $\pg(3,q)$.
Furthermore, in some cases the line classes we construct are related to two-intersection sets in $\ag(2,q)$.
Since there are very few known examples of these sets for $q$ odd, any new results in this direction are of
particular interest.


\end{abstract}

\section{Introduction}\label{intro}
Cameron and Liebler~\cite{CL} studied sets of lines in $\pg(3,q)$ having certain nice properties, today known as
\emph{Cameron-Liebler line classes}.  Such a set contains $x(q^2+q+1)$ lines for some integer $x$, and among many
equivalent properties, shares with every spread of $\pg(3,q)$ precisely $x$ lines; $x$ is called the \emph{parameter}
of the set.  There are some trivial examples of Cameron-Liebler line classes; if we take all of the lines through a
common point, or all of the lines in a common plane, we get a line class with parameter 1.  If we take a
non-incident point-plane pair, and take all of the lines through the point along with all of the lines in the plane,
we get a line class having parameter 2.  Also, it is easy to see that the complement of a Cameron-Liebler line class
with parameter $x$ is a line class with parameter $q^2 +1 -x$.

While it was originally conjectured by Cameron and Liebler that only the trivial examples of these line classes
would exist, this was disproved by an example due to Drudge in~\cite{D}, which was shortly thereafter shown to be
part of an infinite family of examples by Bruen and Drudge~\cite{BD}. These examples exist for all odd values of
$q$, and have $x = \frac{1}{2}(q^2+1)$.  The first counterexample to the conjecture for an even value of $q$ was
given by Goevarts and Penttila in~\cite{GP}, having $q = 4$ and $x = 7$.  There has been much work excluding certain
parameters for Cameron-Liebler line classes;
Govaerts and Storme~\cite{GS} first showed that there are no Cameron-Liebler line classes in $\pg(3,q)$ with
parameter $2 < x \leq q$ when $q$ is prime.  Some time later, De Beule, Hallez, Storme~\cite{BHS} excluded
parameters $2 < x \leq q/2$ for all values of $q$.  Most recently, Metsch~\cite{M} eliminated the possibility of
having $2 < x \leq q$ for any prime power $q$.

In this work, we detail new Cameron-Liebler line classes constructed for many odd values of $q$ satisfying
$q \equiv 1 \bmod{4}$ and $q \not\equiv 1 \bmod{3}$, having parameter $x = \frac{1}{2}(q^2 -1)$.  These new examples
are made up of a union of orbits of a cyclic collineation group having order $q^2 +q +1$.  A Cameron-Liebler line
class with parameter $x$ Klein-corresponds to an $x$-tight set of points in $Q^{+}(5,q)$~\cite{DT}; we can use this
model to easily test if a given set of points in $Q^{+}(5,q)$ corresponds to a Cameron-Liebler line class using an
eigenvector condition.  Among these new examples, when $q \equiv 0 \bmod{3}$ we seem to have some especially
interesting behavior; in these cases, the Cameron-Liebler line class constructed gives rise to a symmetric tactical
decomposition of $\pg(3,q)$, which can be used to derive a set of type $(m,n)$ in the affine plane $\ag(2,q)$.  Our
hope is that these new Cameron-Liebler line classes will belong to an infinite family, and that there will also be
an infinite family of related affine sets of type $(m,n)$.  Affine sets of type $(m,n)$ have previously only been
found in planes of order $9$, so the discovery of an infinite family of new examples would hold particular interest.

\section{Cameron-Liebler line classes (tight sets of $Q^{+}(5,q)$)}\label{sec:1}
\begin{definition}
  A \textbf{Cameron-Liebler line class}  $\mathcal{L}$ is a set of lines in $\pg(3,q)$ such that  any line $\ell$ of
  $\pg(3,q)$ is incident with
  \begin{equation*}
    |\{\ m \in \mathcal{L} : \ \mbox{$m$ meets $\ell$, $m \neq \ell$} \}| =
    \left\lbrace
      \begin{array}{ll}
        (q+1)x + (q^2 -1) & \mbox{ if $\ell \in \mathcal{L}$}\\
        (q+1)x & \mbox{ if $\ell \not\in \mathcal{L}$}
      \end{array}
    \right.
  \end{equation*}
  for a fixed integer $x$, called the \emph{parameter} of $\mathcal{L}$.
\end{definition}
There are many other equivalent characterizations of these sets of lines; for an extensive list see~\cite{Pe}.  The
lines of $\pg(3,q)$ Klein-correspond to points of $Q^{+}(5,q)$~\cite{DT}.  The Klein correspondence is a bijection
(which we shall always denote by $\kappa$) between the set of lines of $\pg(3, q)$ and the set points of
$Q^{+}(5, q)$, such that lines $\ell$ and $\ell^{\prime}$ of $\pg(3, q)$ are incident if and only if $\kappa(\ell)$
and $\kappa(\ell^{\prime})$ are collinear.  In this context, we will be interested in \emph{tight sets} of points.
The notion of an $i$-tight set of a finite generalized quadrangle was introduced by Payne in~\cite{P}, and was
extended to polar spaces of higher rank by Drudge in~\cite{DT}.
\begin{definition}
  A set of points $\mathcal{T}$ in $Q^{+}(5,q)$ is said to be \textbf{$i$-tight} if
  \begin{equation*}
    \left| P^{\perp} \cap \mathcal{T} \right| = \left\lbrace
      \begin{array}{ll}
        i (q+1) + q^{2} & \mbox{ if } P \in \mathcal{T} \\
        i (q+1)             & \mbox{ if } P \not\in \mathcal{T} \mbox{.}
      \end{array}
    \right.
  \end{equation*}
\end{definition}
It is easy to see that a Cameron-Liebler line class of $\pg(3,q)$ with parameter $x$ Klein-corresponds to an
$x$-tight set of $Q^{+}(5,q)$.

We now describe the model of the Klein quadric we will be using.
Let $F = GF(q)$, $E = GF(q^3)$, and $T$ be the relative trace function from $E$ to $F$.   We consider the quadric
$Q^{+}(5,q)$ to have the underlying vector space $V = E^2$ considered as a vector space over $F$, and equipped with
the quadratic form $Q : (x,y) \to T(xy)$.  The polar form of $Q$ is given by
$f((u_{1},u_{2}), \ (v_{1}, v_{2})) = T(u_{1}v_{2}) + T(u_{2}v_{1})$.  This form is nondegenerate, and it can be
seen that $\pi_{1} = \{(x,0) : \ x \in E^{*}\}$ and $\pi_{2} = \{(0,y) : \ y \in E^{*}\}$ are totally singular
planes in the quadric.

We will be using the orbits of a cyclic group to construct our Cameron-Liebler line classes.  Take $\mu \in E^{*}$
with $|\mu| = q^2 +q + 1$.  Define the map $g$ on $Q^{+}(5,q)$ by
$g: \langle (x,y) \rangle \to \langle (\mu x, \mu^{-1} y) \rangle$.  Now $g$ is a projective isometry, and the group
$C = \langle g \rangle$  can be seen to have $|C| = q^2 +q +1$. This group also stabilizes the planes $\pi_{1}$ and
$\pi_{2}$.
\begin{theorem}
If $q \not\equiv 1 \bmod{3}$, the group $C$ acts semi-regularly on the points of $Q^{+}(5,q)$.
\end{theorem}
\begin{proof}
Notice that $g^{i}\langle (x,y) \rangle = \langle (x,y) \rangle$ implies that $\mu^{i} \in F$.
But $(q^2 +q +1, q-1) = (q-1, 3) = 1$, since $q \not\equiv 1 \bmod{3}$.  Thus we have that $\mu^{i} = 1$, and so
$g^{i}$ is the identity map.
\end{proof}
This model of the Klein quadric and this cyclic group are very nice to work with algebraically; they were used
in~\cite{PW} to construct cyclic parallelisms.  For us, since each orbit has size $q^2 +q +1$, taking $x$ orbits
gives the right amount of points to have an $x$-tight set, so we will seek to combine these orbits in an appropriate
way so that the proper conditions are satisfied.

\section{New line classes}\label{sec:2}
The main tool we will use to find tight sets of $Q^{+}(5,q)$ is an eigenvector condition given in~\cite{BKLP}.
\begin{theorem}
Let $A$ be the collinearity matrix of $Q^{+}(5,q)$ and let $\mathcal{L}$ be an $x$-tight set with characteristic
vector $\chi$.  Then
\begin{equation}
\left(\chi - \frac{x}{q^2 + 1}\mathbf{j}\right)
\end{equation}
is an eigenvector for $A$ with eigenvalue $q^{2} - 1$, where $\mathbf{j}$ is the vector consisting of all ones.
\end{theorem}
This theorem follows directly from our characterization of $i$-tight sets.

We construct new examples of Cameron-Liebler line classes using the computational software MAGMA~\cite{MAG}.  We
begin by taking a distinguished plane $\pi$ in $\pg(3,q)$ along with a distinguished point $P \not\in \pi$.  Thus
$\mbox{line}(\pi) \cup \mbox{star}(P)$ is a Cameron-Liebler line class of parameter $2$, and the complement of this
set is a Cameron-Liebler line class with parameter $q^2 -1$.  Working in the Klein quadric $Q^{+}(5,q)$ as detailed
in the previous section, we can assume that $\mbox{line}(\pi)$ Klein-corresponds to
$\pi_{1} = \{(x,0) : \ x \in E^{*}  \}$ and
$\mbox{star}(P)$ corresponds to $\pi_{2} = \{(0,y) : \ y \in E^{*}  \}$.  Defining our group $C$ as before, and
requiring that $q \not\equiv 1 \bmod{3}$, we have $\pi_{1}$ and $\pi_{2}$ as orbits of $C$, as well as $q^2 -1$
other orbits each having size $q^2 +q +1$.  Our wish is to split these other orbits in half in such a way that we
obtain two Cameron-Liebler line classes $\mathcal{L}_{1}$ and $\mathcal{L}_{2}$ each having parameter
$\frac{1}{2}(q^2 -1)$.

The difficulty in searching for these line classes is twofold.  First, constructing the collinearity matrix $A$ of
$Q^{+}(5,q)$ is very time consuming.  Second, combining orbits through brute force is not at all computationally
efficient, and searching the eigenspace of $A$ corresponding to $(q^2 -1)$ for eigenvectors of the appropriate form
is not much better.  To circumvent this first problem, we avoid constructing the entire collinearity matrix of
$Q^{+}(5,q)$.  Instead, we use the following result~\cite{God}:
\begin{lemma}
  Suppose $A$ can be partitioned as
  \begin{equation}
    A =
    \left[
      \begin{array}{ccc}
        A_{11}   & \cdots & A_{1k} \\
        \vdots & \ddots & \vdots \\
        A_{k1}   & \cdots & A_{kk}
      \end{array}
    \right]
  \end{equation}
with each $A_{ii}$ square, $1 \leq i \leq k$, and each $A_{ij}$ having constant row sum $b_{ij}$.  Then any
eigenvalue of the matrix $B = (b_{ij})$ is also an eigenvalue of $A$.
\end{lemma}
The nice thing about applying this lemma is that we can easily construct an eigenvector of $A$ from an eigenvector
of $B$.  By using the point orbits of a group which we assume stabilizes the Cameron-Liebler line class we wish to
construct to partition the points of $Q^{+}(5,q)$, we can construct a matrix $B$ more easily than constructing $A$,
and an appropriate eigenvector of $B$ corresponding to the eigenvalue $(q^2-1)$ will give rise to a tight point set
of $Q^{+}(5,q)$.

Let $q \equiv 1 \bmod{4}$ and $q \not\equiv 1 \bmod{3}$.  Assume we have an $\frac{1}{2}(q^2-1)$-tight set
$\mathcal{L}_{1}$ of $Q^{+}(5,q)$, disjoint from $\pi_{1}$ and $\pi_{2}$, which is stabilized by $C$ as well as the
maps $\sigma : (x,y) \mapsto (x^q, y^q)$ and $\theta: (x,y) \mapsto (x, \omega^4 y)$, where
$\langle \omega \rangle = F^{*}$.  We should notice that
$\mathcal{L}_{2} = Q^{+}(5,q) \setminus (\pi_{1} \cup \pi_{2} \cup \mathcal{L}_{1})$ is also a
$\frac{(q^2-1)}{2}$-tight set.  By using the orbits of the group
$G = C\langle \sigma \rangle \langle \theta \rangle$ (which has order $\frac{3}{4}(q-1)(q^2+q+1)$) to partition the
points of $Q^{+}(5,q)$, we obtain a matrix $B$ that is $\frac{(q+1)}{3} \times \frac{(q+1)}{3}$ when
$q \equiv 2 \bmod{3}$, and $(\frac{q}{3}+1)\times(\frac{q}{3}+1)$ when $q \equiv 0 \bmod{3}$.  We have applied this
technique for all prime powers $q < 200$ satisfying $q \equiv 1 \bmod{4}$ and $q \not\equiv 1 \bmod{3}$, and in each
case we are able find eigenvectors corresponding to $\frac{(q^2-1)}{2}$-tight sets (one unique example up to
isomorphism for each $q$). In each case, $G$ ends up being the full stabilizer of the constructed set.\\

\section{Affine two-intersection sets}\label{sec:4}
A set of type $(m,n)$ in a projective or affine plane is a set $\mathcal{K}$ of points such that every line of the
plane contains either $m$ or $n$ points of $\mathcal{K}$;  we require that $m < n$, and we want both values to
occur.  For projective planes, there are many examples of these types of sets with $q$ both even and odd.  However,
the situation is quite different for affine planes.  When $q$ is even, we obtain a set of type $(0,2)$ in $\ag(2,q)$
from a hyperoval of the corresponding projective plane, and similarly a set of type $(0,n)$ from a maximal arc.
Examples of sets of type $(m,n)$ in affine planes of odd order, on the other hand, are extremely scarce.  The only
previously known examples exist in affine planes of order $9$~\cite{PR}, where we have sets of type $(3,6)$
containing either $36$ or $45$ points (these are complementary cases).  These examples were found by a computer
search.

A Cameron-Liebler line class $\mathcal{L}_{1}$ of parameter $40$ in $\pg(3,9)$ constructed in the previous section
induces a symmetric tactical decomposition on the space having four classes of points and lines. The four line
classes are $\mbox{line}(\pi)$, $\mbox{star}(P)$, $\mathcal{L}_{1}$, $\mathcal{L}_{2}$, where
$\mathcal{L}_{2} = Q^{+}(5,q) \setminus (\pi_{1} \cup \pi_{2} \cup \mathcal{L}_{1})$, and the four point classes are
$\pi$, $P$, and two others, $\mathcal{P}_{1}$ and $\mathcal{P}_{2}$.  $\mathcal{P}_{1}$ and $\mathcal{P}_{2}$ are
obtained as follows: each point in the space, excluding $P$ and the points of $\pi$, lies on either $30$ or $60$
lines of $\mathcal{L}_{1}$.  We define $\mathcal{P}_{1}$ to be the set of points on $30$ lines of $\mathcal{L}_{1}$,
and $\mathcal{P}_{2}$ to be the set of points on $60$ lines of $\mathcal{L}_{1}$.  Now, if we take a plane
$\pi^{\prime}$ of $\pg(3,9)$ not equal to $\pi$, and not containing $P$, $\pi^{\prime}$ contains precisely one line
of $\mbox{line}(\pi)$, no lines of $\mbox{star}(P)$, and either $30$ or $60$ lines of $\mathcal{L}_{1}$, so $60$ or
$30$ lines of $\mathcal{L}_{2}$.  WLOG we may assume that $\pi^{\prime}$ contains $30$ lines of $\mathcal{L}_{1}$
and $60$ lines of $\mathcal{L}_{2}$.  $\pi^{\prime}$ also contains $(q+1)$ points of $\pi$, and, under our
assumptions, $30$ points of $\mathcal{P}_{1}$ and $60$ points of $\mathcal{P}_{2}$.  In fact, this is a symmetric
tactical decomposition of $\pi^{\prime}$ having $3$ classes on points and lines.  Finally, by taking
$\pi^{\prime} \cap \pi$ to be the line at $\infty$, we derive the affine plane $\ag(2,9)$.  It can be easily verified
that $\pi^{\prime} \cap \mathcal{P}_{1}$ is a set of size $30$ in $\ag(2,9)$ that is of type $(3,6)$.  As the sets of
type $(m,n)$ in $\ag(2,9)$ were completely classified in~\cite{PR}, this set is not new.  However, following this
same procedure with our Cameron-Liebler line class of parameter $3280$ in $\pg(3,81)$ yields an affine set of type
$(36, 45)$ in $\ag(2,81)$, which is new.  Our hope is that, if related Cameron-Liebler line classes exist in
$\pg(3, 3^{2e})$ for all $e$, they will always give rise to sets of type $(m,n)$ in $\ag(2, 3^{2e})$.

\bibliography{CLlineclasses}

\end{document}